\documentclass[10pt,a4paper]{amsart}

\theoremstyle{plain}
\newtheorem{theorem}{Theorem}[section]
\newtheorem{lemma}[theorem]{Lemma}
\newtheorem{proposition}[theorem]{Proposition}
\newtheorem{corollary}[theorem]{Corollary}
\theoremstyle{definition}

\theoremstyle{remark}
\newtheorem{remark}[theorem]{Remark}

\newcommand{\hc}{\ensuremath{\mathcal{H}}}

\newcommand{\bQ}{\mathbb{Q}}

\def\bin #1#2 {\left( \matrix { #1 \cr #2 \cr } \right) }

\begin{document}

\title[On the topology of a resolution of  isolated singularities]
{On the topology of a resolution of isolated singularities}

\author{Vincenzo Di Gennaro }
\address{Universit\`a di Roma \lq\lq Tor Vergata\rq\rq, Dipartimento di Matematica,
Via della Ricerca Scientifica, 00133 Roma, Italy.}
\email{digennar@axp.mat.uniroma2.it}

\author{Davide Franco }
\address{Universit\`a di Napoli
\lq\lq Federico II\rq\rq, Dipartimento di Matematica e
Applicazioni \lq\lq R. Caccioppoli\rq\rq, P.le Tecchio 80, 80125
Napoli, Italy.} \email{davide.franco@unina.it}

\abstract Let $Y$ be a complex projective  variety of dimension $n$
with isolated singularities, $\pi:X\to Y$ a resolution of
singularities, $G:=\pi^{-1}{\rm{Sing}}(Y)$ the exceptional locus.
From Decomposition Theorem  one knows that the map $H^{k-1}(G)\to
H^k(Y,Y\backslash {\rm{Sing}}(Y))$ vanishes for $k>n$. Assuming this
vanishing, we give a short proof of Decomposition Theorem for $\pi$.
A consequence is a short proof of the Decomposition Theorem for
$\pi$ in all cases where one can prove the vanishing directly. This
happens when either $Y$ is a normal surface, or when $\pi$ is the
blowing-up of $Y$ along ${\rm{Sing}}(Y)$ with smooth and connected
fibres, or when $\pi$ admits a natural Gysin morphism. We prove that this
last condition is equivalent to say that the map $H^{k-1}(G)\to
H^k(Y,Y\backslash {\rm{Sing}}(Y))$ vanishes for any $k$, and that
the pull-back $\pi^*_k:H^k(Y)\to H^k(X)$ is injective. This provides
a relationship between Decomposition Theorem and Bivariant Theory.



\endabstract
\maketitle

\bigskip
\section{Introduction}

Consider a $n$-dimensional complex projective variety $Y$   with
 \textit{isolated singularities}. Fix a {\it desingularization} $\pi: X\rightarrow Y$ of $Y$.
This paper is addressed at the study of some topological
properties of the map $\pi $. In a previous paper \cite{DGF1} we
already observed that, even though $\pi$ is never a \textit{local
complete intersection} map, in some very special case it may
anyway admit a \textit{natural Gysin morphism}.  By natural Gysin
morphism we mean a \textit{topological bivariant class}  \cite[\S
7]{FultonCF},  \cite{DeCM}
$$\theta \in T^0(X\stackrel{\pi}{\rightarrow}Y):=Hom_{D^b(Y)}(R\pi_*\mathbb Q_X , \mathbb Q_Y),
$$
commuting with restrictions to the smooth locus of $Y$ (here and
in the following $D^b(Y)$ denotes the \textit{bounded derived
category} of sheaves of $\mathbb Q$-vector spaces on $Y$).

In this paper we give a complete characterization of morphisms
like $\pi $ admitting a natural Gysin morphism by means of the
\textit{Decomposition Theorem} \cite{BBD}, \cite{DeCM1}, \cite{DeCM2}, \cite{DeCM3}.
In some sense, what we are going
to prove is that  $\pi$ admits a natural Gysin morphism if and only if $Y$ is
a $\bQ$-\textit{intersection cohomology manifold}, i.e.
$IC^{\bullet}_Y\simeq \bQ_Y[n]$ in $D^b(Y)$ ($IC^{\bullet}_Y$
denotes \textit{intersection cohomology complex} of $Y$ \cite[p.
156]{Dimca2}, \cite{Massey}). Furthermore, in this case, there is
a unique natural Gysin morphism $\theta$, and it arises from the
Decomposition Theorem (compare with Theorem \ref{premain} below).

The Decomposition Theorem
is a beautiful and very deep result
about algebraic maps. In the words of MacPherson \lq\lq it
contains as special cases the deepest homological properties of
algebraic maps that we know\rq\rq  \cite{Mac83},
\cite{Williamson}. As observed in \cite[Remark 2.14]{Williamson},
since the proof of the Decomposition Theorem proceeds by induction
on the dimension of the strata of the singular locus, a key point
of such a Theorem is the case of varieties with isolated
singularities:

\begin{theorem}[Decomposition Theorem for varieties with isolated singularities]
\label{DecTh} In $D^b(Y)$ we have a decomposition
$$R\,\pi_*\mathbb Q_X\cong IC^{\bullet}_Y[-n]\oplus \mathcal H^{\bullet}$$
where $\mathcal H^{\bullet}$ is quasi isomorphic to a skyscraper
complex on ${\rm{Sing}}(Y)$ and
\begin{enumerate}
\item $\hc^k(\mathcal H^{\bullet})\cong H^k(G)$, for any $k\geq n$,
\item $\hc^k(\mathcal H^{\bullet})\cong H_{2n-k} (G)$, for any $k< n$,
\end{enumerate}
where we set $G:=\pi ^{-1}({\rm{Sing}}(Y))$.
\end{theorem}

The relationship between Gysin morphism and Decomposition Theorem
is mostly related to an important topological property of the
morphism $\pi$. Specifically, in  \cite{GMtm} and
\cite{Steenbrink} it is showed that Theorem \ref{DecTh} implies
the following vanishing

\begin{equation}
\label{top_hyp} H^{k-1}(G)\to H^k(Y,U) \hskip2mm
\textrm{vanishes for} \hskip2mm k>n.
\end{equation}

One of the main points we would like to stress in this paper
(compare with Theorem \ref{main1}) is that

\bigskip
\begin{center}\label{equivalence}
\textit{the vanishing  (\ref{top_hyp}) is  equivalent  to the
Decomposition Theorem.}
\end{center}
\bigskip

More precisely, what we are going to do in this paper is to prove
that assuming (\ref{top_hyp}), one can prove Theorem \ref{DecTh}
in few pages. Actually this equivalence is already implicit in the
argument developed by Navarro Aznar in order to prove \cite[(6.3)
Corollaire, p. 293]{N}. In fact, after proving  (\ref{top_hyp})
using Hodge Theory, in \cite{N} one proves relative Hard Lefschetz
Theorem and concludes thanks to Deligne's Theorems on degeneration
of spectral sequences. Instead, here we give a more simple and
direct proof, without using Hard Lefschetz Theorem. In fact we
deduce the splitting in derived category by a simple result about
short exact sequences of complexes (compare with Lemma \ref{sec}).


A byproduct of our result is a short proof of the Decomposition
Theorem in all cases where one can prove property (\ref{top_hyp})
directly. This happens when either $2\dim G< n$ (for trivial
reasons), or when $Y$ is a {\it normal surface} in view of {\it
Mumford's Theorem} \cite{Ishii}, \cite{Mumford}, or when $\pi:X\to
Y$ is the {\it blowing-up} of $Y$ along ${\rm{Sing}}(Y)$ with smooth and
connected fibres (see Remark \ref{rmain1}). It is worth remarking
that if $Y$ is { locally complete intersection} then {\it Milnor's
Theorem} on the connectivity of the {\it link} \cite{Dimca1} implies
(via Lemma \ref{uno} below) that  the map $H^{k-1}(G)\to H^{k}(Y,U)$
vanishes for any $k\geq n+2$. Therefore in this case the question
reduces to only check that the map $H^{n}(G)\to H^{n+1}(Y,U)$
vanishes. This in turn is equivalent to require that $H_n(G)$, which
is contained in $H_n(X)$ via push-forward, is a non degenerate
subspace of $H_n(X)$ with respect to the natural intersection form
$H_n(X)\times H_n(X)\to H_0(X)$ (see Remark \ref{rmain1}, (i)).
Another case is when $\pi $ admits a Gysin morphism. Indeed, in this
case it is very easy to prove the stronger property

$$
H^{k-1}(G)\to H^k(Y,U) \hskip2mm
\textrm{vanishes for} \hskip2mm k>0.
$$

This is the real reason why in our approach the same line of
arguments leads to both Theorem \ref{DecTh} and and the following:

\begin{theorem}\label{premain}
There exists a natural Gysin morphism for $\pi $ if and only if
$Y$ is a $\mathbb Q$-intersection cohomology manifold. In this
case, in $D^b(Y)$ we have a decomposition
$$R\,\pi_*\mathbb Q_X\,\cong\,IC^{\bullet}_Y[-n]\oplus \mathcal H^{\bullet}\,\cong\,
\mathbb Q_Y \oplus \bigoplus_{k\geq 1} R^k\pi_*\mathbb Q_X[-k].
$$
Moreover a natural Gysin morphism is unique, and, up to
multiplication by a nonzero rational number, it comes from the
decomposition above via projection onto $\mathbb Q_Y$.
\end{theorem}

For a more precise and complete  statement see Theorem \ref{main2}
and Remark \ref{unicity} below. For instance, from Theorem
\ref{main2}, (ix), we see that natural Gysin morphisms occur when
$Y$ is nodal of even dimension $n$, or when $Y$ is a cone over a
smooth basis $M$ with $H^{\bullet}(M)\cong H^{\bullet}(\mathbb
P^{n-1})$. We stress that the existence of a natural Gysin
morphism forces the exceptional locus $G$ to have dimension $0$ or
$n-1$ (see Remark \ref{fibre'}).

Last but not least, we have been led to consider the issues
addressed in this paper by our previous work on Noether-Lefschetz
Theory. We refer to the papers \cite{DGF},
\cite{IJM}, \cite{DFM}, \cite{RCMP} anyone interested in the
overlaps between the topological properties investigated here and
Noether-Lefschetz Theorem (specifically, we made an heavy use of
Decomposition Theorem in \cite[Remark 3 and Theorem 6, (6.3), p.
169]{DFM}, and in \cite[Theorem 2.1, proof of (a), p. 262]{RCMP}).

\section{Notations}

$(i)$ Let $Y$ be a {complex irreducible projective variety} of
dimension $n\geq 1$, with {isolated singularities}. Let $\pi:X\to Y$
be a {resolution of the singularities} of $Y$. For any $y\in
\text{Sing}(Y)$ set $G_y:=\pi^{-1}(y)$. Set $G:=\bigcup_{y\in
\text{Sing}(Y)} G_y=\pi^{-1}(\text{Sing}(Y))$. Let
$i:G\hookrightarrow X$ be the inclusion.

\medskip
$(ii)$ All cohomology and homology groups are with $\mathbb
Q$-coefficients.

\medskip
$(iii)$ Set $U:=Y\backslash \text{Sing}(Y)\cong X\backslash G$.
Denote by $\alpha:U\hookrightarrow Y$ and $\beta:U\hookrightarrow
X$ the inclusions. For any integer $k$ we have the following
natural commutative diagram:
\begin{equation}\label{t}
\begin{array}{ccccc}
H^k(Y) &\stackrel{\pi^*_k}{\longrightarrow}  & H^k(X)\\
\quad\stackrel{\alpha^*_k}{}\searrow&\,&\swarrow\stackrel{\beta^*_k}{}\\
&H^k(U)&\\
\end{array}
\end{equation}
where all the maps denote pull-back.

\begin{remark}\label{triangolo} From commutativity of diagram (\ref{t}) we get
$\Im(\alpha^*_k)\subseteq \Im(\beta^*_k)$. Since  $H^k(Y)\cong
H^k(X)$ for $k\leq 0$ or $k\geq 2n$, we have
$\Im(\alpha^*_k)=\Im(\beta^*_k)$ for $k\leq 0$ or $k\geq 2n$. It
may happen that $\Im(\alpha^*_k)\neq \Im(\beta^*_k)$. We may
interpret the condition $\Im(\alpha^*_k)=\Im(\beta^*_k)$ as
follows. From Universal Coefficient Theorem and Lefschetz Duality
Theorem \cite[p. 248 and p. 297]{Spanier} we have $H^k(U)\cong
H_{2n-k}(Y,\text{Sing}(Y))$ for any $k$. Since $\text{Sing}(Y)$ is
finite we also have $H_{2n-k}(Y)\cong H_{2n-k}(Y,\text{Sing}(Y))$
for $k\leq 2n-2$, and $H_1(Y)\subseteq H_{1}(Y,\text{Sing}(Y))$.
Therefore, for $k\leq 2n-2$, diagram (\ref{t}) identifies with the
diagram:
$$
\begin{array}{ccccc}
H^k(Y) &\stackrel{}{\longrightarrow}  & H_{2n-k}(X)\\
\quad\stackrel{}{}\searrow&\,&\swarrow\stackrel{}{}\\
&H_{2n-k}(Y)&\\
\end{array}
$$
where the map $H^k(Y) \stackrel{}{\to}H_{2n-k}(X)$ is the
composite of Poincar\'e Duality $H^k(X)\cong H_{2n-k}(X)$ with the
pull-back $\pi^*_k$, the map $H_{2n-k}(X)\to H_{2n-k}(Y)$ is the
push-forward, and the map $H^k(Y)\stackrel{\cdot\,\cap
[Y]}{\longrightarrow} H_{2n-k}(Y)$ is the {\it duality morphism}, i.e.
the cap-product with the fundamental class $[Y]\in H_{2n}(Y)$
\cite{McCrory}. It follows that {\it
$\Im(\alpha^*_k)=\Im(\beta^*_k)$ if and only if any cycle in
$H_{2n-k}(Y)$ coming from $H_{2n-k}(X)$ via push-forward is  the
cap-product of a cocycle in $H^k(Y)$ with the fundamental class
$[Y]$}. This holds true also for $k=2n-1$ because $H_1(Y)\subseteq
H_{1}(Y,\text{Sing}(Y))\cong H^{2n-1}(U)$.
\end{remark}

\medskip
$(iv)$ Embed $Y$ in some projective space $\mathbb P^N$. For any
$y\in \text{Sing}(Y)$ choose a small closed ball $S_y\subset
\mathbb P^N$ around $y$, and set $B_y:=S_y\cap Y$,
$D_y:=\pi^{-1}(B_y)$, $B:=\bigcup_{y\in \text{Sing}(Y)} B_y$, and
$D:=\pi^{-1}(B)$. $B_y$ is homeomorphic to the cone over the {
link} $\,\partial B_y$ of the singularity $y\in Y$, with vertex at
$y$ \cite[p. 23]{Dimca1}. $B_y$ is contractible, by excision we
have $H^k(Y,U)\cong H^k(B,B\backslash \text{Sing}(Y))\cong
H^k(B,\partial B)$ for any $k$, and from the cohomology long exact
sequence of the pair $(B,\partial B)$ we get $H^k(Y,U)\cong
H^{k-1}(\partial B)$ for any $k\geq 2$. We have $\partial D\cong
\partial B$ via $\pi$, and  by excision we have
$H^k(X,U)\cong H^k(D,D\backslash G)\cong H^k(D,\partial D)$ for
any $k$ \cite[p. 38]{Dimca2}. Since $G$ is homotopy equivalent to
$D$, we  have $H^k(G)\cong H^k(D)$. Putting all together, from the
cohomology long exact sequence of the pair $(D,\partial D)$ we get
the following exact sequence
\begin{equation}\label{es1}
H^k(X,U)\stackrel{}{\to} H^k(G)\to
H^{k+1}(Y,U)\stackrel{\gamma^*_{k+1}}{\to} H^{k+1}(X,U)
\end{equation}
for any $k\geq 1$, where $\gamma^*_{k+1}$ denotes the pull-back.
Observe that since $\text{Sing}(Y)$ is finite we have
$H^k(G)=\oplus_{y\in \text{Sing}(Y)} H^k(G_y)$,
$H^k(B)=\oplus_{y\in \text{Sing}(Y)} H^k(B_y)$, $H^k(\partial B
)=\oplus_{y\in \text{Sing}(Y)} H^k(\partial B_y)$.

\begin{remark}\label{lci}
Assume $Y$ is {locally complete intersection}. In this case,
from the connectivity of the link  \cite[Milnor's Theorem p. 76, and
Hamm's Theorem  p. 80]{Dimca1}, it follows that {\it the duality
morphism $H^k(Y)\to H_{2n-k}(Y)$ is an isomorphism for any
$k\notin\{n-1,n,n+1\}$, is injective for $k=n-1$, and is
surjective for $k=n+1$. In particular
$\Im(\alpha^*_k)=\Im(\beta^*_k)$ for any $k\notin\{n-1,n\}$}. In
order to prove this property, we argue as follows. We may assume
$0<k<2n$ and $n\geq 2$. From the cohomology long exact sequence of the pair
$(Y,U)$ we have:
\begin{equation}\label{ess2}
\dots {\to}\, H^{k}(Y,U)\to H^{k}(Y)\to H^k(U)\to H^{k+1}(Y,U)\to
\dots,
\end{equation}
and by excision $H^k(Y,U)\cong H^k(B,\partial B)$. Taking into
account that each $B_y$ is contractible and that $\partial B_y$ is
path connected  \cite[loc. cit.]{Dimca1}, from the cohomology long
exact sequence of the pair $(B,\partial B)$ we get $H^1(B,\partial
B)=0$ and $H^k(B,\partial B)\cong H^{k-1}(\partial B)$ for $k\geq
2$. Since $H^k(U)\cong H_{2n-k}(Y,\text{Sing}(Y))$, and
$H_{2n-k}(Y)\cong H_{2n-k}(Y,\text{Sing}(Y))$ for $k\leq 2n-2$,
from (\ref{ess2}) we get the exact sequence for
$k\notin\{1,2n-1\}$ (compare with \cite[p. 5]{DGF2}):
$$
H^{k-1}(\partial B)\to H^{k}(Y)\to H_{2n-k}(Y)\to H^{k}(\partial
B).
$$
Each $\partial B_y$ is $(n-2)$-connected by Milnor's Theorem
\cite[loc. cit.]{Dimca1}, and it is a compact oriented real
manifold of dimension $2n-1$, in particular $h^k(\partial
B_y)=h^{2n-1-k}(\partial B_y)$ by Poincar\'e Duality \cite[p.
91]{Dimca1}. It follows that the map $H^{k}(Y)\to H_{2n-k}(Y)$ is
an isomorphism for $k\notin\{1,n-1,n,n+1,2n-1\}$. As for the case
$k=1\neq n-1$, this follows from (\ref{ess2}) because $H^{1}(Y,U)\cong
H^1(B,\partial B)=0$, $H^1(U)\cong H_{2n-1}(Y,\text{Sing}(Y))\cong
H_{2n-1}(Y)$, and $H^2(Y,U)\cong H^2(B,\partial B)\cong
H^1(\partial B)=0$ by connectivity of the link. When $k=2n-1\neq n+1$ we
have $H^{2n-1}(Y,U)\cong H^{2n-1}(B,\partial B)=H^{2n-2}(\partial
B)=0$, therefore $H^{2n-1}(Y)\hookrightarrow H^{2n-1}(U)$. On the
other hand $H_1(Y)\hookrightarrow H_1(Y,\text{Sing}(Y))\cong
H^{2n-1}(U)$. It follows that the duality morphism $H^{2n-1}(Y)\to
H_1(Y)$ is injective. Then it is an isomorphism because we have
just seen, in the case $k=1$, that $h^1(Y)=h_{2n-1}(Y)$. Finally
notice that, when $n\geq 3$, from previous analysis and
(\ref{ess2}) we get the exact sequence:
$$
0\to H^{n-1}(Y)\to H_{n+1}(Y)\to H^{n-1}(\partial B)\to
H^{n}(Y)\to H_{n}(Y)
$$
$$
\to H^{n}(\partial B)\to H^{n+1}(Y)\to H_{n-1}(Y)\to 0.
$$
Therefore the duality morphism $H^{n-1}(Y)\to H_{n+1}(Y)$ is
injective, and $H^{n+1}(Y)\to H_{n-1}(Y)$ is onto. This holds true
also when $n=2$. In fact also in this case we have
$ H^{1}(B,\partial B)=0$, which implies that
the duality morphism $H^{1}(Y)\to H_{3}(Y)$ is
injective. Moreover a similar analysis as before shows that
the image of $H^3(Y)$  and $H_1(Y)$ have the same codimension in
$H^3(U)$, and therefore they are equal. This concludes
the proof of the claim.
\end{remark}

\medskip
$(v)$ By \cite[Lemma 14, p. 351]{Spanier} we have $H^k(X,U)\cong
H_{2n-k}(G)$. Therefore from the cohomology long exact sequence of
the pair $(X,U)$ we get a long exact sequence:
\begin{equation}\label{es2}
\dots {\to}\, H^{k-1}(U)\to H_{2n-k}(G)\to
H^k(X)\stackrel{\beta^*_k}\to H^k(U)\to \dots.
\end{equation}

\medskip
$(vi)$ For any $y\in \text{Sing}(Y)$ set:
$$
H^k_y:= \begin{cases} H^k(G_y) \quad{\text{if $k\geq n$}} \\
H_{2n-k}(G_y) \quad{\text{if $k<n$.}} \end{cases}
$$
Let $\mathcal H^k_y$ be the skyscraper sheaf on $Y$ with stalk at
$y$ given by $H^k_y$. Set $H^k:=\oplus_{y\in \text{Sing}(Y)}H^k_y$
and $\mathcal H^k:=\oplus_{y\in \text{Sing}(Y)} \mathcal H^k_y$.
We consider $\mathcal H^{\bullet}$ as a
complex of sheaves on $Y$ with vanishing differentials
$d^k_{\mathcal H^{\bullet}}=0$.

\begin{remark}\label{hself} By Universal Coefficient Theorem \cite[p. 248
]{Spanier} it follows that the $\mathbb Q$-vector spaces
${H^{n-k}}$ and ${H^{n+k}}$ are isomorphic for any $k$. This
implies that {\it $\mathcal H^{\bullet}[n]$ is self-dual, i.e. in
the bounded derived category $D^b(Y)$ of $Y$ we have $\mathcal
H^{\bullet}[n]\cong D(\mathcal H^{\bullet}[n])$}.  Taking into
account that in $\mathcal H^{\bullet}[n]$ all the differentials
vanish, to prove that $\mathcal H^{\bullet}[n]$ is self-dual it
suffices to prove that the complexes $\mathcal H^{\bullet}[n]$ and
$D(\mathcal H^{\bullet}[n])$ have isomorphic sheaf cohomology.
Since $\mathcal H^{\bullet}[n]$ is supported on a finite set, this
amounts to prove that $\mathcal H^{\bullet}[n]$ and $D(\mathcal
H^{\bullet}[n])$ have isomorphic hypercohomology, i.e. that
$$
\mathbb H^k(\mathcal H^{\bullet}[n])\cong \mathbb H^k(D(\mathcal
H^{\bullet}[n]))
$$
for any $k$. But by Poincar\'e-Verdier Duality \cite[p. 69,
Theorem 3.3.10]{Dimca2} we have:
$$
\mathbb H^k(D(\mathcal H^{\bullet}[n]))\cong \mathbb
H^{-k}(\mathcal H^{\bullet}[n])^\vee \cong \mathbb
H^{n-k}(\mathcal H^{\bullet})^\vee \cong ({H^{n-k}})^\vee \cong
{H^{n+k}} \cong \mathbb H^{k}(\mathcal H^{\bullet}[n]).
$$
\end{remark}

\medskip
$(vii)$ We say that a graded morphism
$\theta_{\bullet}:H^{\bullet}(X)\to H^{\bullet}(Y)$ is {\it
natural} if for any $k$ one has $\theta_{k}\circ
\pi^*_k={\text{id}}_{H^k(Y)}$, and the following diagram commutes
\cite{DGF1}:
$$
\begin{array}{ccccc}
H^k(Y) &\stackrel{\theta_k}{\longleftarrow}  & H^k(X)\\
\quad\stackrel{\alpha^*_k}{}\searrow&\,&\swarrow\stackrel{\beta^*_k}{}\\
&H^k(U),&\\
\end{array}
$$
i.e. $\alpha^*_k\circ \theta_k=\beta^*_k$.

\medskip
\begin{remark}\label{natural} The existence of a natural graded
morphism $\theta_{\bullet}:H^{\bullet}(X)\to H^{\bullet}(Y)$ is
equivalent to say that, for any $k$, the pull-back
$\pi^*_k:H^k(Y)\to H^k(X)$ is injective and $\Im(\alpha^*_k)=
\Im(\beta^*_k)$ (compare with the proof of (i) $\implies$ (ii) in
Theorem \ref{main2} below).
\end{remark}

\medskip
$(viii)$ We say that a (topological) {bivariant} class $\theta
\in Hom_{D^b(Y)}(R\pi_*\mathbb Q_X, \mathbb Q_Y)$ is {\it natural}
if the induced graded morphism $\theta_{\bullet}:H^{\bullet}(X)\to
H^{\bullet}(Y)$ is natural \cite{DGF1}, \cite{FultonCF}.

\begin{remark}\label{bivariant} Fix any bivariant class $\theta
\in H^0(X\stackrel{\pi}\to Y)\cong  Hom_{D^b(Y)}(R\pi_*\mathbb
Q_X, \mathbb Q_Y)$. Let $\theta_0:H^0(X)\to H^0(Y)$ be the induced
map. Let $q\in\mathbb Q$ be such that $\theta_0(1_X)=q\cdot 1_Y\in
H^0(Y)\cong \mathbb Q$ \cite[p. 238]{Spanier}. Put
$$
\deg \theta:=q.
$$
For any $k$ and any $c\in H^k(Y)$, by the projection formula
\cite[(G$_4$), (i), p. 26]{FultonCF}, and \cite[9, p.
251]{Spanier}, we have :
\begin{equation}\label{bivdue}
\theta_k(\pi^*_k(c))=\theta_k(1_X\cup
\pi^*_k(c))=\theta_0(1_X)\cup c=\deg \theta\cdot (1_Y\cup c)=\deg
\theta\cdot c.
\end{equation}
It follows that for any $k$ one has:
\begin{equation}\label{bivuno}
\theta_{k}\circ \pi^*_k=\deg \theta\cdot {\rm{id}}_{H^k(Y)}.
\end{equation}
Next consider the independent square:
$$
\begin{array}{ccccc}
U &\stackrel{\beta}{\hookrightarrow}  &  X\\
\stackrel{}{}\Vert&  &\stackrel{\pi}{}\downarrow\\
U& \stackrel{\alpha}{\hookrightarrow} &Y\\
\end{array}
$$
and set $\theta':=\alpha^*(\theta)\in Hom_{D^b(U)}(\mathbb Q_U,
\mathbb Q_U)$ \cite[(G$_2$), p. 26]{FultonCF}. Applying
\cite[(G$_2$), (ii), p. 26]{FultonCF} to the square:
$$
\begin{array}{ccccc}
H^0(U) &\stackrel{\beta^*_0}{\leftarrow}  &  H^0(X)\\
\stackrel{\theta'_0}{}\downarrow&  &\stackrel{\theta_0}{}\downarrow\\
H^0(U)& \stackrel{\alpha^*_0}{\leftarrow} &H^0(Y)\\
\end{array}
$$
we get
$$
\theta'_0(1_U)=\theta'_0(\beta^*_0(1_X))=\beta^*_0(\theta_0(1_X))=\beta^*_0(\deg
\theta\cdot 1_Y)=\deg\theta\cdot\beta^*_0(1_Y)=\deg\theta\cdot
1_U.
$$
Since $\pi_{|_U}= {\rm{id}}_{U}$, as in (\ref{bivdue})  we deduce
for any $k$ and any $c\in H^k(U)$:
$$
\theta'_k(c)=\theta'_k(({\pi{|_U}})^*_k(c))=\theta'_k(1_U\cup
c)=\theta'_0(1_U)\cup c=\deg \theta\cdot (1_U\cup c)=\deg
\theta\cdot c,
$$
i.e.
\begin{equation}\label{bivquattro}
\theta'_k=\deg\theta \cdot{\rm{id}}_{H^k(U)}.
\end{equation}
From \cite[(G$_2$), (ii), p. 26]{FultonCF} it follows that
\begin{equation}\label{bivtre}
\deg\theta \cdot \beta^*_k=\theta'_k\circ\beta^*_k=\alpha^*_k\circ
\theta_k
\end{equation}
for any $k$. By (\ref{bivuno}) and (\ref{bivtre}) we see that {\it
a bivariant class $\theta$ is natural if and only if $\deg
\theta=1$, and this is equivalent to say that $\beta^*_k=
\alpha^*_k\circ \theta_k$ for any $k$}. Observe that if $\theta$
is any bivariant class with $\deg\theta\neq 0$, then
$\frac{1}{\deg\theta}\theta$ is natural.
\end{remark}

\medskip
$(ix)$ We say that $Y$ is a {\it $\mathbb Q$-cohomology (or
homology) manifold} if for any $y\in Y$ and any $k\neq 2n$ one has
$H^k(Y,Y\backslash\{y\})=0$, and $H^{2n}(Y,Y\backslash\{y\})\cong
\mathbb Q$ \cite{Massey}, \cite{McCrory}. Recall  that $Y$ is a
{$\mathbb Q$-{intersection cohomology manifold}} if
$IC^{\bullet}_Y\cong \bQ_Y[n]$ in $D^b(Y)$, where $IC^{\bullet}_Y$
denotes the {intersection cohomology complex} of $Y$ \cite[p.
156]{Dimca2}, \cite{Massey}.

\begin{remark}\label{dimuno} By \cite[3.1.4, p. 34]{FultonCF}
we know that there is a mapping $\phi:X\to \mathbb R^m$ such that
$(\pi, \phi):X\to Y\times \mathbb R^m$ is a closed imbedding. In
this case one has
$$
H^0(X\stackrel{\pi}\to Y)\cong H^m(Y\times \mathbb R^m,Y\times
\mathbb R^m\backslash X_{\phi}),
$$
where $X_{\phi}$ is the image of $X$ in $Y\times \mathbb R^m$. If
$Y$ is a $\mathbb Q$-cohomology manifold, then by Poincar\'e-Alexander-Lefschetz Duality
\cite[Theorem 1.1]{AFP} we have:
$$
H^m(Y\times \mathbb R^m,Y\times \mathbb R^m\backslash
X_{\phi})\cong H_{2n}(X).
$$
It follows that
\begin{equation}\label{bivcinque}
\dim_{\mathbb Q} H^0(X\stackrel{\pi}\to Y)=1.
\end{equation}
On the other hand, since $U$ is smooth, we also have \cite[Lemma 2
and (26), p. 217]{FultonYT}:
$$
H^0(U\stackrel{{\rm{id}}_U}\to U)\cong H^m(U\times \mathbb
R^m,U\times \mathbb R^m\backslash U_{\phi})\cong
H^{BM}_{2n}(U)\cong H^0(U)\cong \mathbb Q,
$$
where $H^{BM}_{2n}(U)$ denotes Borel-Moore homology. Therefore the
pull-back
$$
\alpha^*:H^0(X\stackrel{\pi}\to Y)\to
H^0(U\stackrel{{\rm{id}}_U}\to U)
$$
for bivariant classes identifies with the restriction in
Borel-Moore homology:
$$
H_{2n}(X)\cong H^{BM}_{2n}(U).
$$
Comparing with (\ref{bivquattro}) and (\ref{bivcinque}), this
proves that {\it if $Y$ is a $\mathbb Q$-cohomology manifold then
there is a unique natural bivariant class}.
\end{remark}

\medskip
$(x)$ Let $\mathcal I^{\bullet}$ be an injective resolution of
$\mathbb Q_X$. Let $\mathcal J^{\bullet}:=\pi_*(\mathcal
I^{\bullet})$ be the derived direct image $R\,\pi_*\mathbb Q_X$ of
$\mathbb Q_X$ in $D^b(Y)$. When $k\geq 1$ the cohomology sheaves
$R^k \pi_*\mathbb Q_X=H^k(\mathcal J^{\bullet})$ are supported on
$\text{Sing}(Y)$, and for any $y\in \text{Sing}(Y)$ we have
$H^k(\mathcal J^{\bullet})_y=H^k(G_y)$.

\begin{remark}\label{jself} The complex $\mathcal J^{\bullet}[n]$
is self-dual. In fact by \cite[p. 69, Proposition 3.3.7,
(ii)]{Dimca2} we have:
$$
D(\mathcal J^{\bullet}[n])=D(R\pi_*\mathbb
Q_X[n])=R\pi_*(D(\mathbb Q_X[n]))=R\pi_*(\mathbb Q_X[n])=\mathcal
J^{\bullet}[n].
$$
\end{remark}

\medskip
$(xi)$ Since $Y$ has only isolated singularities, we have
\cite[Proposition 5.4.4, p. 157]{Dimca2}:
\begin{equation}\label{cic}
IH^k(Y)\cong \begin{cases} H^k(Y) \quad{\text{if $k>n$}}\\
\Im(\alpha^*_n) \quad{\text{if $k=n$}}\\ H^k(U) \quad{\text{if
$k<n$.}} \end{cases}
\end{equation}

\section{The main results}

Theorem \ref{main1} below is essentially already known. Property
(i) implies (ii) by \cite[Theorem 1.11, p. 518]{Steenbrink}. That
property (ii) implies (i) is implicit in the argument developed by
Navarro in order to prove \cite[(6.3) Corollaire, p. 293]{N} using
a relative version of Hard Lefschetz Theorem. Here we give a more
simple and direct proof that (ii) implies (i), without using Hard
Lefschetz Theorem.

\begin{theorem}
\label{main1} The following properties are equivalent.

\smallskip
(i) In the derived category of $Y$ there is an isomorphism
$R\pi_*\mathbb Q_X\cong IC_Y[n]\oplus \mathcal H^{\bullet}$.

\smallskip
(ii) The map $H^{k-1}(G)\to H^k(Y,U)$ vanishes for any
$k> n$.
\end{theorem}

\medskip
The equivalence of properties (v), (vi) and (vii) in next
Theorem  \ref{main2} are already known \cite{BM}, \cite{McCrory},
\cite{Massey}. We insert them in the claim for Reader's
convenience. We refer to \cite{Massey} for other equivalence
concerning a $\mathbb Q$-cohomology manifold.

\begin{theorem}
\label{main2} The following properties are equivalent.

\smallskip
(i) The map $H^{k-1}(G)\to H^k(Y,U)$ vanishes for any
$k>0$ and the pull-back $\pi^*_k$ is injective.

\smallskip
(ii) There exists a natural graded morphism
$\theta_{\bullet}:H^{\bullet}(X)\to H^{\bullet}(Y)$.

\smallskip
(iii) There exists a natural bivariant class $\theta \in
Hom_{D^b(Y)}(R\pi_*\mathbb Q_X, \mathbb Q_Y)$.

\smallskip
(iv)  The natural map ${H^{\bullet}(Y)}\to{IH^{\bullet}(Y)}$is an
isomorphism;

\smallskip
(v)  $Y$ is a $\mathbb Q$-intersection cohomology manifold.

\smallskip
(vi)  $Y$ is a $\mathbb Q$-cohomology manifold.

\smallskip
(vii) The duality morphism $H^{\bullet}(Y)\stackrel{\cdot\,\cap
[Y]}\longrightarrow H_{2n-\bullet}(Y)$ is an isomorphism (i.e. $Y$
satisfies Poincar\'e Duality).

\smallskip
(viii)  In $D^b(Y)$ there exists a decomposition
\begin{equation}\label{eldec}
R\,\pi_*\mathbb Q_X\cong \mathbb Q_Y\oplus \bigoplus_{k\geq 1}
R^k\pi_*\mathbb Q_X[-k].
\end{equation}

\smallskip\noindent
Moreover, if $\pi:X\to Y$ is the blowing-up of $Y$ along
${\rm{Sing}}(Y)$ with smooth and connected fibres, then previous
properties are equivalent to the following property:

\smallskip
(ix) For any $y\in \text{Sing}(Y)$ one has
${H^{\bullet}(G_y)}\cong {H^{\bullet}(\mathbb P^{n-1})}$.
\end{theorem}

\begin{remark}
\label{unicity} (i) Projecting onto $\mathbb Q_Y$, from the
decomposition  (\ref{eldec}) we obtain a bivariant class $\eta \in
Hom_{D^b(Y)}(R\pi_*\mathbb Q_X, \mathbb Q_Y)$, whose induced Gysin
morphisms $\eta_k:H^k(X)\to H^k(Y)$ are surjective. In particular
$\deg\eta\neq 0$. By Remark \ref{dimuno} it follows  that
$\frac{1}{\deg\eta}\eta$ is the unique natural bivariant class.

\smallskip
(ii) The natural morphism $\theta_{\bullet}:H^{\bullet}(X)\to
H^{\bullet}(Y)$ is unique and identifies with the push-forward via
Poincar\'e Duality:
$$
H^{\bullet}(X)\cong H_{2n-\bullet}(X)\to H_{2n-\bullet}(Y)\cong
H^{\bullet}(Y).
$$
In fact by Remark \ref{triangolo} we know that, for $k<2n-1$, the
restriction map $\alpha^*_k:H^k(Y)\to H^k(U)$ is nothing but the
duality (iso)morphism because $H^k(U)\cong H_{2n-k}(Y)$. Therefore
$\theta_k={(\alpha^*_k)}^{-1}\circ \beta^*_k$. The case $k=2n-1$
is similar because $H_1(Y)\subseteq H^{2n-1}(U)$ (again compare
with Remark \ref{triangolo}).
\end{remark}

\section{Preliminaries}

\begin{lemma}\label{uno} The following sequences
are exact:
$$
0\to H^k(Y)\stackrel {\pi^*_k}\to H^k(X) \stackrel {i^*_k}\to
H^k(G)\to 0 \quad {\text{for any $k>n$,}}
$$
$$
H^n(Y)\stackrel {\pi^*_n}\to H^n(X) \stackrel {i^*_n}\to H^n(G)\to
0,
$$
$$
0\to H_{2n-k}(G)\to
H^k(X)\stackrel{\beta^*_k}\to H^k(U)\to 0 \quad
{\text{for any $k<n$.}}
$$
\end{lemma}

\begin{proof} By \cite[p. 84, $6^*$]{Dold} we know that $H^{k}(Y,\text{Sing}(Y))\cong
H^{k}(X,G)$ for any $k$. Since $\text{Sing}(Y)$ is finite, we also
have $H^{k}(Y,\text{Sing}(Y))\cong H^k(Y)$ for $k\geq 1$.
Therefore the long exact sequence of the pair:
$$
{\dots} \to H^{k}(X,G)\stackrel
{}{\to}H^{k}(X)\stackrel{i^*_k}{\to} H^{k}(G)\stackrel
{}{\to}H^{k+1}(X,G)\to \dots
$$
identifies, when $k\geq 1$, with the long exact sequence:
\begin{equation}\label{es4}
{\dots} \to H^{k}(Y)\stackrel
{\pi^*_k}{\to}H^{k}(X)\stackrel{i^*_k}{\to} H^{k}(G)\stackrel
{}{\to}H^{k+1}(Y)\to \dots.
\end{equation}
In order to prove that the first two sequences are exact, it
suffices to prove that $i^*_k$ is surjective for any $k\geq n$. To
this aim let $L$ be a general hyperplane section of $Y$, and put
$Y_0:=Y\backslash L$, and $X_0:=\pi^{-1}(Y_0)$. As before, we have
a long exact sequence:
$$
{\dots} \to H^{k}(Y_0)\stackrel {}{\to}H^{k}(X_0)\stackrel{}{\to}
H^{k}(G)\stackrel {}{\to}H^{k+1}(Y_0)\to \dots
$$
and by the Deligne's Theorem \cite[Proposition 4.23]{Voisin} we know
that the pull-back maps $H^{k}(X)\stackrel{i^*_k}{\to} H^{k}(G)$ and
$H^{k}(X_0)\stackrel{}{\to} H^{k}(G)$ have the same image. Then we
are done, because $Y_0$ is affine, therefore $H^{k+1}(Y_0)=0$ for
any $k\geq n$ by stratified Morse Theory \cite[p. 23-24]{GM}.

In order to examine the last sequence, assume $k<n$. Then
$2n-k>n$, and we just proved that the pull-back
$H^{2n-k}(X,G)\cong H^{2n-k}(Y)\to H^{2n-k}(X)$ is injective. By
Poincar\'e Duality Theorem and Lefschetz Duality Theorem \cite[p.
297]{Spanier} we have $H^{2n-k}(X)\cong H_{k}(X)$ and
$H^{2n-k}(X,G)\cong H_k(U)$. Therefore the push-forward $H_k(U)\to
H_k(X)$ is injective, hence the restriction $H^k(X)\to H^k(U)$ is
onto for any $k<n$. Now our assertion follows from (\ref{es2}).
\end{proof}

\begin{lemma}\label{due}
Fix an integer $k$, and let $\gamma^*_k:H^k(Y,U)\to H^k(X,U)$ be
the pull-back. Assume that $\pi^*_k:H^k(Y)\to H^k(X)$ is
injective. Then the following properties are equivalent.

\smallskip
(i) $\gamma^*_k$ is injective;

\smallskip
(ii) $\Im(\alpha^*_{k-1})=\Im(\beta^*_{k-1})$;

\smallskip
(iii) $H^{k-1}(G)\to H^{k}(Y,U)$ is the zero map.
\end{lemma}

\begin{proof} Consider the natural commutative diagram with exact
rows:
$$
\begin{array}{ccccccc}
H^{k-1}(X)&\stackrel
{\beta^*_{k-1}}{\longrightarrow}&H^{k-1}(U)&\stackrel{}{\longrightarrow}
&H^{k}(X,U)&\stackrel {}{\longrightarrow}&H^{k}(X)\\
\stackrel {\pi^*_{k-1}}{}\uparrow & &\stackrel {}{}\Vert & & \stackrel {\gamma^*_k}{}\uparrow& &
\stackrel {\pi^*_k}{}\uparrow \\
H^{k-1}(Y)&\stackrel
{\alpha^*_{k-1}}{\longrightarrow}&H^{k-1}(U)&\stackrel{}{\longrightarrow}
&H^{k}(Y,U)&\stackrel {}{\longrightarrow}&H^{k}(Y).\\
\end{array}
$$
If $\gamma^*_k$ is injective then
$$
\ker(H^{k-1}(U)\to H^k(X,U))=\ker(H^{k-1}(U)\to H^k(Y,U)).
$$
It follows that $\Im(\alpha^*_{k-1})=\Im(\beta^*_{k-1})$ because
$\Im(\alpha^*_{k-1})=\ker(H^{k-1}(U)\to H^k(Y,U))$ and
$\Im(\beta^*_{k-1})=\ker(H^{k-1}(U)\to H^k(X,U))$. Conversely,
assume that $\Im(\alpha^*_{k-1})=\Im(\beta^*_{k-1})$, and fix any
$c\in \ker \gamma^*_k$. Since $\pi^*_k$ is injective, there exists
some $c'\in H^{k-1}(U)$ which maps to $c$ via $H^{k-1}(U)\to
H^k(Y,U)$. Since $c\in \ker \gamma^*_k$, a fortiori $c'$ belongs
to  $\Im(\beta^*_{k-1})$. Hence $c'\in \Im(\alpha^*_{k-1})$,
therefore $c=0$. The equivalence of (i) with (iii) follows from
(\ref{es1}).
\end{proof}


\begin{corollary}\label{hl}
Let $H_{k}(G)\to H^{2n-k}(G)$ be the map obtained
composing the map $H_{k}(G)\to H^{2n-k}(X)$ with the
pull-back $H^{2n-k}(X)\to H^{2n-k}(G)$. Assume $k\geq
n$ and that  $\Im(\alpha^*_{k})=\Im(\beta^*_{k})$. Then the map
$H_{k}(G)\to H^{2n-k}(G)$ is injective.
\end{corollary}

\begin{proof} By Lemma \ref{uno}, Lemma \ref{due}, and  (\ref{es1}),
we deduce that the map $H^{k}(X,U)\to H^{k}(G)$ is
onto. Dualizing we get an injective map $H_{k}(G) \to
H_{k}(X,U)$. We are done because by excision and Lefschetz Duality
Theorem \cite[p. 298]{Spanier} we have $ H_{k}(X,U)\cong
H_{k}(D,\partial D)\cong H^{2n-k}(D)\cong H^{2n-k}(G)$.
\end{proof}

\begin{corollary}\label{ic} We have:
$$
H^k(X)\cong \begin{cases} IH^k(Y)\oplus H^k(G) \quad{\text{if $k>n$}}\\
 IH^k(Y)\oplus H_{2n-k}(G) \quad{\text{if $k<n$.}}
\end{cases}
$$
Moreover if $\Im(\alpha^*_{n})=\Im(\beta^*_{n})$ then
$$
H^n(X)\cong IH^n(Y)\oplus H^n(G).
$$
\end{corollary}

\begin{proof} In view of Lemma \ref{uno} we only have to examine
the case $k=n$. Since $\beta^*_n\circ \pi^*_n=\gamma^*_n$, there
exists a subspace $P\subseteq \Im(\pi^*_n)\subseteq H^n(X)$ which
is mapped isomorphically to
$\Im(\beta^*_{n})=\Im(\alpha^*_{n})=IH^n(Y)$ via $\beta^*_n$. In
particular $P\cap \ker \beta^*_n=\{0\}$, and so
$H^n(X)=IH^n(Y)\oplus \ker \beta^*_n$. On the other hand $\ker
\beta^*_n=\Im(H^n(X,U)\to H^n(X))$. By Corollary \ref{hl} we know
that the map $H^n(X,U)\to H^n(X)$ is injective because so is the
composite $H^n(X,U)\cong H_{n}(G)\to
H^{n}(X)\to H^{n}(G)$. Therefore $\ker
\beta^*_n=\Im(H^n(X,U)\to H^n(X))\cong H^n(X,U)\cong H_{n}(G)\cong
H^n(G)$.
\end{proof}

\begin{lemma}\label{inc} Assume that $\Im(\alpha^*_{k})=\Im(\beta^*_{k})$
for any $k\geq n$. Then there is an injective map of complexes
$$
0\to \mathcal H^{\bullet}\to \mathcal J^{\bullet}.
$$
\end{lemma}

\begin{proof} It is enough to prove that for any $k$ there is a
monomorphism of sheaves $\mathcal H^k \hookrightarrow \ker\,(
\mathcal J^k\to \mathcal J^{k+1})$.

First we examine the case $k\geq n$.

To this aim, set $\Gamma^{\bullet}:= \Gamma(\mathcal J^{\bullet})$
and denote by $d^k: \Gamma^k\to \Gamma^{k+1}$ the differential.
Then we have $H^k(X) = H^k(\Gamma^{\bullet})$. By Lemma \ref{uno}
any element $a$ of $H^k=H^k(G)$ can be lifted to an element $c \in
\ker\,d^k$. We claim that any $a\in H^k(G)$ can be lifted to an
element $b\in \ker\,d^k\subseteq \Gamma (\mathcal J^k)$ which is
supported on $\text{Sing}(Y)$. Proving this claim amounts to show
that any $a\in H^k(G)$ can be lifted to an element $b\in
\ker\,d^k\subset \Gamma (\mathcal J^k)=\Gamma (\mathcal I^k)$ such
that $b\mid_U=0\in \Gamma (\mathcal J^k\mid_U)$. But $c\mid_U$
projects to a cohomology class living in $\Im(H^k(X)\to H^k(U))$.
By our assumption we have
$$\Im(H^k(X)\stackrel{\beta^*_k}\to H^k(U))\subseteq
\Im(H^k(Y)\stackrel{\alpha^*_k}\to H^k(U)).
$$
Since
$$
H^k(Y)\cong H^k(Y, \text{Sing}(Y))\cong H^k(X,G)
$$
\cite[p. 84, $6^*$]{Dold}, we find
$$
\Im(H^k(Y)\stackrel{\alpha^*_k}\to H^k(U))=
\Im(H^k(X,G)\to H^k(U)).
$$
On the other hand we have
$$
H^k(X,G)\cong H^k(X,{\beta}_!\mathbb Q _U)$$ \cite[Theorem
12.1]{Br}, \cite[Remark 2.4.5, (ii)]{Dimca2}. By definition of
direct image with proper support \cite[\S 2.6]{Iversen},
\cite[Definition 2.3.21]{Dimca2}, the sheaf ${\beta}_!\mathbb Q_U$
identifies with the subsheaf of $\mathbb Q_X$ consisting of
sections with support contained in $U$. It follows there exists
$e_U\in \Gamma(\mathcal J^{k-1}\mid_U)$ and $ g\in \Gamma(\mathcal
J^k)$ supported in $U$ such that
$$c\mid_U-d^{k-1}(e_U)=g\mid_U.$$
Moreover  there exists $e\in \Gamma(\mathcal J^{k-1})$ with
$e\mid_U=e_U$, because $\mathcal J^{k-1}$ is injective (hence
flabby). We conclude that the section
$$c- g -d^{k-1}(e)\in \Gamma(\mathcal J^k)
$$
is supported on $\text{Sing}(Y)$. Our claim is proved because $g
+d^{k-1}(e)\in \Gamma(\mathcal J^k)$ vanishes in $H^k(G)$. To
conclude the proof in the case $k\geq n$, fix a basis $a_r\in
H^k=H^k(G)$ and lift any $a_r$ to a $b_r\in \ker\, d^k\subseteq
\Gamma (\mathcal J^k)$ as in the claim. We get an isomorphism
between $H^k(G)$ and a subspace of $\Gamma (\mathcal J^k)$
consisting of sections supported on $\text{Sing}(Y)$. We are done
because such an isomorphism projects to a monomorphism of sheaves
$\mathcal H^k \hookrightarrow \ker\,(J^k\to J^{k+1})$.

Now assume $k<n$.

By Lemma \ref{uno} any element $a$ of $H^k=H_{2n-k}(G)\subseteq
H^k(X)$ can be lifted to an element $c \in \ker\,d^k$. Since $a$
restricts to $0$ in $H^k(U)$, then there exists $e\in
\Gamma(\mathcal J^{k-1}\mid_U)$ such that  $c\mid_U=d^{k-1}_U(e)$.
Since $\mathcal J^{k-1}$ is flabby we may assume $e\in
\Gamma(\mathcal J^{k-1})$. Therefore $b:=c-d^{k-1}(e)\in
\Gamma(\mathcal J^{k})$ represents $a$ and is supported on
$\text{Sing}(Y)$. As in the case $k\geq n$, applying this argument
to a basis of $H^k=H_{2n-k}(G)$, we define a monomorphism of
sheaves $\mathcal H^k \hookrightarrow \ker\,(\mathcal J^k\to
\mathcal J^{k+1})$.
\end{proof}

With the same assumption as in Lemma \ref{inc}, let $\mathcal
K^{\bullet}$ be the cokernel of the  inclusion $0\to \mathcal
H^{\bullet}\to \mathcal J^{\bullet}$:
$$
0\to \mathcal H^{\bullet}\to \mathcal J^{\bullet}\to \mathcal
K^{\bullet}\to 0.
$$
All the sheaves of these complexes are injective. Previous
sequence gives rise to a long exact sequence of sheaf cohomology:
$$
\dots {\to}\, \mathcal H^k \to \mathcal H^k(J^{\bullet})\to
\mathcal H^k(K^{\bullet})\to \dots,
$$
and for any $k\geq 1$ these sheaves are supported on
$\text{Sing}(Y)$.

\begin{proposition}\label{tat} For any $k$ the sequence
$$
0 {\to}\,\, \mathcal H^k\to \mathcal H^k(J^{\bullet})\to \mathcal
H^k(K^{\bullet})\to 0
$$
is exact.
\end{proposition}

\begin{proof} It suffices to prove that  the map $H^k_y\to
\mathcal H^k(\mathcal J^{\bullet})_y$ is injective for any $y\in
\text{Sing}(Y)$ and any $k>0$. If $k\geq n$ this is obvious because
$H^k(\mathcal J^{\bullet})_y=H^k(G_y)=H^k_y$. When $1\leq k<n$ we
have $H^k_y=H_{2n-k}(G_y)$. And the map $H_{2n-k}(G_y)\to
H^k(\mathcal J^{\bullet})_y=H^k(G_y)$ is injective by Corollary
\ref{hl}.
\end{proof}

\begin{lemma}\label{sec} Let
$0\to \mathcal H^{\bullet}\stackrel{f^{\bullet}}\to \mathcal
J^{\bullet}\stackrel{g^{\bullet}}\to \mathcal K^{\bullet}\to 0$ be
an exact sequence of complexes of sheaves. Assume that $\mathcal
H^{\bullet}$ is a complex of injective sheaves with vanishing
differential $d^{k}_{\mathcal H^{\bullet}}=0$ for any $k$. The
following properties are equivalent.

\smallskip
(i) The sequence coming from the cohomology long exact sequence:
\begin{equation}\label{es3}
0 {\to} \mathcal H^k(\mathcal H^{\bullet})\to \mathcal
H^k(\mathcal J^{\bullet})\to \mathcal H^k(\mathcal K^{\bullet})\to
0
\end{equation}
is exact for any $k$.

\smallskip
(ii) There is a complex map $s^{\bullet}:\mathcal K^{\bullet}\to
\mathcal J^{\bullet}$ such that $g^{\bullet}\circ
s^{\bullet}={\rm{id}}_{\mathcal K^{\bullet}}$.
\end{lemma}

\begin{proof} We only have to prove that (i) implies (ii).

Since $\mathcal H^0$ is injective, the exact sequence sequence $0
{\to} \mathcal H^0\to \mathcal J^0\to \mathcal K^0\to 0$ admits a
section $s^0:\mathcal K^0\to \mathcal J^0$, with $g^{0}\circ
s^{0}={\text{id}}_{\mathcal K^{0}}$. Therefore we may construct
$s^{\bullet}=\{s^i\}_{i\geq 0}$ using induction on $i$. Assume
$i\geq 0$ and that there are sections $s^0,\dots,s^i$, with
$s^h:\mathcal K^h\to \mathcal J^h$, $g^{h}\circ
s^{h}={\text{id}}_{\mathcal K^{h}}$, and $s^h\circ
d^{h-1}_{\mathcal K^{\bullet}}= d^{h-1}_{\mathcal
J^{\bullet}}\circ s^{h-1}$ for any $0\leq h\leq i$. As before,
since $\mathcal H^{i+1}$ is injective and the sequence $0 {\to}
\mathcal H^{i+1}\to \mathcal J^{i+1}\to \mathcal K^{i+1}\to 0$ is
exact, there exists a section $\sigma^{i+1}:\mathcal K^{i+1}\to
\mathcal J^{i+1}$, with $g^{i+1}\circ
\sigma^{i+1}={\text{id}}_{\mathcal K^{i+1}}$. A priori it may
happen that $\sigma^{i+1}\circ d^{i}_{\mathcal K^{\bullet}}$ is
different from $d^{i}_{\mathcal J^{\bullet}}\circ s^{i}$, so we
have to modify $\sigma^{i+1}$. To this purpose set:
$$
\delta:=\sigma^{i+1}\circ d^{i}_{\mathcal K^{\bullet}}-
d^{i}_{\mathcal J^{\bullet}}\circ s^{i}\in Hom(\mathcal K^i,\mathcal
J^{i+1}).
$$
Since
$$
g^{i+1}\circ \delta=g^{i+1}\circ \sigma^{i+1}\circ d^{i}_{\mathcal
K^{\bullet}} -g^{i+1}\circ d^{i}_{\mathcal J^{\bullet}}\circ
s^{i}=d^{i}_{\mathcal K^{\bullet}}-d^{i}_{\mathcal K^{\bullet}}=0,
$$
it follows that
\begin{equation}\label{puno}
\Im(\delta)\subseteq \mathcal H^{i+1}.
\end{equation}
Moreover, since (\ref{es3}) is exact, the map $g^i$ sends $\ker
d^{i}_{\mathcal J^{\bullet}}$ onto $\ker  d^{i}_{\mathcal
K^{\bullet}}$. This in turn implies that the section $s^i$ sends
$\ker d^{i}_{\mathcal K^{\bullet}}$ in $\ker d^{i}_{\mathcal
J^{\bullet}}$, because
$$
\ker g^i=\Im (f^i)\cong \mathcal H^i=\mathcal H^i({\mathcal
H^{\bullet}})\subseteq \ker d^{i}_{\mathcal J^{\bullet}}
$$
in view of the assumption $d^{i}_{\mathcal H^{\bullet}}=0$. We
deduce that:
\begin{equation}\label{pdue}
\ker d^{i}_{\mathcal K^{\bullet}} \subseteq \ker \delta,
\end{equation}
and from (\ref{puno}) and (\ref{pdue}) we get
$$
\delta\in  Hom(\mathcal K^i\slash \ker d^{i}_{\mathcal
K^{\bullet}},\mathcal H^{i+1}).
$$
Since $\mathcal H^{i+1}$ is injective, we may extend $\delta$ to a
map $\tilde\delta\in Hom(\mathcal K^{i+1},\mathcal H^{i+1})$ such
that
\begin{equation}\label{ptre}
\tilde \delta\circ d^{i}_{\mathcal K^{\bullet}} = \delta.
\end{equation}
We have
$$
\tilde\delta\in Hom(\mathcal K^{i+1},\mathcal J^{i+1})
$$
because $\mathcal H^{i+1}$ maps to $\mathcal J^{i+1}$ via
$f^{i+1}$. Now we define:
$$
s^{i+1}:=\sigma^{i+1}-\tilde\delta.
$$
From (\ref{ptre}) it follows that
$$
s^{i+1}\circ d^{i}_{\mathcal K^{\bullet}}= d^{i}_{\mathcal
J^{\bullet}}\circ s^{i},
$$
and since $\Im(\tilde\delta)\subseteq \mathcal H^{i+1}$ we also
have
$$
g^{i+1}\circ s^{i+1}={\text{id}}_{\mathcal K^{i+1}}.
$$
\end{proof}

\section{Proof of Theorem \ref{main1}}

As we noticed in Section 3,  by \cite[Theorem 1.11, p.
518]{Steenbrink} one knows that the Decomposition Theorem implies
(ii). Therefore we only have to prove that (ii) implies (i).

\smallskip
In view of  Lemma \ref{uno} and  Lemma \ref{due} we have
$\Im(\alpha^*_{k})=\Im(\beta^*_{k})$ for any $k\geq n$. From Lemma
\ref{inc}, Proposition \ref{tat}, and Lemma \ref{sec}, we get:
\begin{equation}\label{split}
R\pi_*\mathbb Q_X=\mathcal J^{\bullet}=\mathcal K^{\bullet}\oplus
\mathcal H^{\bullet}.
\end{equation}
Hence we only have  to prove that
$$
\mathcal K^{\bullet}\cong IC_Y[-n],
$$
where $IC^{\bullet}_Y=IC^{top}_Y[-n]$ denotes the intersection cohomology
complex of $Y$ \cite[p. 156]{Dimca2}. Observe that the restriction
$\alpha^{-1}\mathcal K^{\bullet}$ of $\mathcal K^{\bullet}$ to $U$
is $\mathbb Q_U$, and that, by (\ref{split}), we have $\mathcal
K^{\bullet}\in D^{b}_c(Y)$ \cite[p. 81-82]{Dimca2}. Therefore
$\mathcal K^{\bullet}[n]$ is an extension of $\mathbb Q_U[n]$
\cite[p. 134]{Dimca2}. So to prove that $\mathcal K^{\bullet}\cong
IC_Y[-n]$ it suffices to prove that $\mathcal K^{\bullet}[n]\cong
\alpha_{!*}\mathbb Q_U[n]$, i.e. that $\mathcal K^{\bullet}[n]$ is
the intermediary extension of $\mathbb Q_U[n]$ \cite[p.156 and
p.135]{Dimca2}. By \cite[Proposition 5.2.8, p. 135]{Dimca2}, this
in turn reduces to prove that for any $y\in \text{Sing}(Y)$ the
following two conditions hold true ($i_y:\{y\}\to Y$ denotes the
inclusion):

\medskip
(a) $\mathcal H^ki^{-1}_y \mathcal K^{\bullet}[n]=0$  for any
integer $k\geq 0$;

\medskip
(b) $\mathcal H^ki^{!}_y \mathcal K^{\bullet}[n]=0$ for any
integer $k\leq 0$.

\medskip\noindent
As for condition (a) we notice that \cite[p.130]{Dimca2}:
$$
\mathcal H^ki^{-1}_y \mathcal K^{\bullet}[n]=\mathcal H^k
(\mathcal K^{\bullet}[n])_y=\mathcal H^{k+n} (\mathcal
K^{\bullet})_y,
$$
and $\mathcal H^{k+n} (\mathcal K^{\bullet})_y=0$ because
$\mathcal J^{\bullet}=\mathcal K^{\bullet}\oplus \mathcal
H^{\bullet}$, and $ \mathcal H^{k+n} (\mathcal
J^{\bullet})_y=H^{k+n}(G_y)=\mathcal H^{k+n} (\mathcal
H^{\bullet})_y$ for $k\geq 0$.

For the condition (b), first notice that combining (\ref{split})
with Remarks \ref{hself} and \ref{jself}, we deduce that $\mathcal
K^{\bullet}[n]$ is self-dual. Therefore condition (b) reduces to
(a). In fact we have \cite[p. 130, proof of Lemma 5.1.15]{Dimca2}:
$$
\mathcal H^ki^{!}_y \mathcal K^{\bullet}[n]=\mathcal H^{-k}
(i^{-1}_y D(\mathcal K^{\bullet}[n]))^{\vee} =\mathcal H^{-k}
(i^{-1}_y (\mathcal K^{\bullet}[n]))^{\vee} =\mathcal H^{-k+n}
(\mathcal K^{\bullet})_y^\vee=0
$$
because $k\leq 0$.

\begin{remark}\label{rmain1}
(i) If $n=2$ then the map $H^{k-1}(G)\to H^{k}(Y,U)$ vanishes for
any $k\geq n+2$ for trivial reasons. In view of the connectivity of
the link, combining Remark \ref{lci} with Lemma \ref{uno} and Lemma
\ref{due}, we see that this holds true also when $Y$ is locally
complete intersection. Therefore, either when $n=2$ or when $Y$ is
locally complete intersection, in order to deduce the decomposition
(i) in Theorem \ref{main1}, one only has to check that the map
$H^{n}(G)\to H^{n+1}(Y,U)$ is the zero map. On the other hand, the
vanishing of the map $H^{n}(G)\to H^{n+1}(Y,U)$ is equivalent to
require that the natural map $H_n(G)\to H^n(G)\cong H_n(G)^{\vee}$
is onto (compare with (\ref{es1}), (\ref{es2}), and Corollary
\ref{hl}). Since $H_n(G)$ is contained in $H_n(X)$ via push-forward
(Lemma \ref{uno}), it follows that the map $H_n(G)\to H^n(G)\cong
H_n(G)^{\vee}$ is onto if and only if $H_n(G)$ is a non degenerate
subspace of $H_n(X)$ with respect to the natural intersection form
$H_n(X)\times H_n(X)\to H_0(X)\cong \mathbb Q$. By {Mumford's
Theorem} \cite{Ishii}, \cite{Mumford} we know this holds true when
$Y$ is a normal surface. Therefore, {\it in the case $Y$ is a normal
surface (or when $2\dim G<n$), our Theorem \ref{main1} gives a new
and simplified proof of the Decomposition Theorem for $\pi:X\to Y$}.

\medskip
(ii) Assume that $\pi:X\to Y$ is the blowing-up of $Y$ along
$\text{Sing}(Y)$, with smooth and connected fibres. By Poincar\'e
Duality we have $H_{2n-k}(G_y)\cong H^{k-2}(G_y)$ for any $y\in
\text{Sing}(Y)$. It follows that $H^k(X,U)\cong H_{2n-k}(G)\cong
\oplus_{y\in \text{Sing}(Y)} H_{2n-k}(G_y)\cong \oplus_{y\in
\text{Sing}(Y)}  H^{k-2}(G_y)$. Hence the map $H^k(X,U)\to H^k(G)$
identifies with the map $\oplus_{y\in \text{Sing}(Y)}
H^{k-2}(G_y)\to \oplus_{y\in \text{Sing}(Y)} H^{k}(G_y)$ given, on
each summand $H^{k-2}(G_y)\stackrel{}{\to} H^{k}(G_y)$, by the
self-intersection formula, i.e. by the cup-product with the first
Chern class $c_1(N_y)\in H^2(G_y)$ of the normal bundle $N_y$ of
$G_y$ in $X$. Since $\pi$ is the blowing-up along the finite set
$\text{Sing}(Y)$, the dual normal bundle $N_y^{\vee}\cong \mathcal
O_{G_y}(1)$ is ample for any $y\in \text{Sing}(Y)$. By Hard
Lefschetz Theorem it follows that the map
$H^{k-2}(G_y)\stackrel{}{\to} H^{k}(G_y)$ is onto for any $k\geq n$,
and so also the map $H^k(X,U)\to H^k(G)$ is. By (\ref{es1}), this
implies the vanishing of the map $H^k(G)\to H^{k+1}(Y,U)$.
Therefore, also in this case our Theorem \ref{main1} gives a new and
simplified proof of the Decomposition Theorem for $\pi$.


\medskip
(iii) More generally, assume only that the fibres of $\pi:X\to Y$
are smooth and connected, so that $\pi$ is not necessarily the
blowing-up along $\text{Sing}(Y)$. Using the extension of the Hard
Lefschetz Theorem to bundles of higher rank due to Bloch and
Gieseker \cite{BG}, \cite{L}, with a similar argument as before one
proves that {\it if the dual normal bundle $N_y^{\vee}$ of $G_y$ in
$X$ is ample for any $y\in \text{Sing}(Y)$, then the map $H^k(G)\to
H^{k+1}(Y,U)$ vanishes for any $k\geq n$}. In fact,  set $h_y:=\dim
X-\dim G_y$ for any $y\in \text{Sing}(Y)$. Now the map $H^k(X,U)\to
H^k(G)$ identifies with the map $\oplus_{y\in \text{Sing}(Y)}
H^{k-2h_y}(G_y)\to \oplus_{y\in \text{Sing}(Y)} H^{k}(G_y)$ given,
on each summand $H^{k-2h_y}(G_y)\stackrel{}{\to} H^{k}(G_y)$, by the
cup-product with the top Chern class
$c_{h_y}(N_y)=(-1)^{h_y}c_{h_y}(N_y^{\vee})\in H^{2h_y}(G_y)$ of the
normal bundle $N_y$ of $G_y$ in $X$. And such a map is onto for
$k\geq n$ by the quoted extension of  the Hard Lefschetz Theorem,
because $N_y^{\vee}$ is ample. We refer to \cite[Proposition 2.12
and proof]{DGF2} for examples of resolution of singularities
verifying previous assumptions.
\end{remark}

\section{Proof of Theorem \ref{main2}}

(i) $\implies$ (ii) By Lemma \ref{uno} and Lemma \ref{due} we have
$\Im(\alpha^*_k)= \Im(\beta^*_k)$ for any $k$. Let
$y_1,\dots,y_a,y_{a+1},\dots, y_b$ be a basis  of $H^k(Y)$ such
that $\alpha^*_k y_1,\dots,\alpha^*_k y_a$ is a basis for
$\Im(\alpha^*_k)=\Im(\beta^*_k)$, and $y_{a+1},\dots, y_b$ a basis
for $\ker \alpha^*_k$. Since $\pi^*_k(\ker \alpha^*_k)\subseteq \ker
\beta^*_k$, we may extend $\pi^*_k y_{a+1},\dots,\pi^*_k y_b$ to a
basis $\pi^*_k y_{a+1},\dots,\pi^*_k y_b,x_{b+1},\dots,x_c$ of
$\ker\beta^*_k$. Then $\pi^*_k y_1,\dots,\pi^*_k y_a,\pi^*_k
y_{a+1},\dots,\pi^*_k y_b,x_{b+1},\dots,x_c$ is a basis for
$H^k(X)$. Define $\theta_k:H^k(X)\to H^k(Y)$ setting
$\theta_k(\pi^*_k(y_i)):=y_i$, and $\theta_k(x_i):=0$. Then
$\theta_{\bullet}$ is a natural morphism.

\bigskip
(ii) $\implies$ (i) The existence of a natural morphism implies
that $\pi^*_k$ is injective and
$\Im(\beta^*_k)\subseteq\Im(\alpha^*_k)$ for any $k$. Since in
general we have $\Im(\alpha^*_k)\subseteq \Im(\beta^*_k)$ it
follows that $\Im(\alpha^*_k)= \Im(\beta^*_k)$ for any $k$. By
Lemma \ref{uno} and Lemma \ref{due} we get (i).

\bigskip
(ii) $\implies$ (iv) Since $\pi^*_k$ is injective for any $k$,
using (\ref{es4}) we get a short exact sequence:
$$
0 \to H^{k}(Y)\stackrel
{\pi^*_k}{\to}H^{k}(X)\stackrel{i^*_k}{\to} H^{k}(G)\stackrel
{}{\to}0
$$
for any $k\geq 1$. In particular, for any $k\geq 1$, we have
\begin{equation}\label{d1}
H^k(X)\cong H^k(Y)\oplus H^k(G).
\end{equation}
On the other hand, since $\theta_k\circ
\pi^*_k={\rm{id}}_{H^k(Y)}$, the short exact sequence
$$
0 \to \ker\theta_k\stackrel
{}{\to}H^{k}(X)\stackrel{\theta_k}{\to} H^{k}(Y)\stackrel {}{\to}0
$$
admits $\pi^*_k$ as a section. It follows another decomposition:
\begin{equation}\label{d2}
H^k(X)= \pi^*_k H^k(Y)\oplus \ker\theta_k.
\end{equation}
Comparing (\ref{d1}) with (\ref{d2}) we see that
$$
\ker \theta_k\cong H^k(G)
$$
for any $k\geq 1$. On the other hand since $\alpha^*_k\circ
\theta_k=\beta^*_k$ we have
\begin{equation}\label{nucleo}
\ker\theta_k\subseteq \ker(H^{k}(X)\stackrel{\beta^*_k}{\to}
H^{k}(U))=\Im(H^{k}(X,U)\stackrel{}{\to} H^{k}(X)).
\end{equation}
Since $H^{k}(X,U)\cong H_{2n-k}(G)$ it
follows that
\begin{equation}\label{fibre}
\dim H^k(G)\leq \dim H_{2n-k}(G)
\end{equation}
for any $k\geq 1$. By Universal-coefficient formula \cite[p. 248
]{Spanier} we deduce that, for $1\leq k\leq 2n-1$,
\begin{equation}\label{poincare}
\ker\theta_k\cong H^k(G)\cong H_{2n-k}(G).
\end{equation}

Taking into account that $\Im(\alpha^*_n)=\Im(\beta^*_n)$, combining
(\ref{d1}),  (\ref{poincare}) and Corollary \ref{ic}, it follows that
$\dim H^k(Y)=\dim IH^k(Y)$ for any $k$. Therefore, by (\ref{cic}), it suffices to prove
that $\alpha^*_k:H^k(Y)\to H^k(U)$ is surjective for any $k<n$. To this purpose notice
that, for $k<n$, $\beta^*_k$ is surjective by Lemma \ref{uno}. This implies that also $\alpha^*_k$ is
by   (\ref{d2}) and (\ref{nucleo}) (compare with diagram  (\ref{t})).

\bigskip
(iv) $\implies$ (vii) Since intersection cohomology verifies
Poincar\'e Duality \cite[p. 158]{Dimca2}, we have:
$$
H^h(Y)=IH^h(Y)=(IH^{2(m+1)-h}(Y))^{\vee}=(H^{2(m+1)-h}(Y))^{\vee}=H_{2(m+1)-h}(Y).
$$

\bigskip
(vii) $\implies$ (iv) This follows from (\ref{cic}) and Remark
\ref{triangolo}.

\bigskip
(v) $\iff$ (vi) $\iff$ (vii) By \cite[Theorem 2, Lemma 2, Lemma
3]{McCrory} we know that the duality morphism is an isomorphism if
and only if $Y$ is a $\mathbb Q$-cohomology manifold, which is
equivalent to say that $Y$ is a $\mathbb Q$-intersection
cohomology manifold by \cite[Theorem 1.1]{Massey} (compare also
with \cite{BM}).

\bigskip
(vii) $\implies$ (ii) Denote by $d^Y_k:H^k(Y)\to H_{2n-k}(Y)$ the
duality isomorphism, by  $d^X_k:H^k(X)\cong H_{2n-k}(X)$ the
Poincar\'e Duality isomorphism, by $\pi_{*,k}: H_{2n-k}(X)\to
H_{2n-k}(Y)$ the push-forward. Set $\theta_k: H^{k}(X)\to
H^{k}(Y)$ with
$$
\theta_k:=(d^Y_k)^{-1}\circ \pi_{*,k}\circ d^X_k.
$$
Then $\theta_{\bullet}$ is a natural morphism.

\bigskip
(iii) $\iff$ (ii) We only have to prove that (ii) implies (iii).
This follows from Remark \ref{dimuno} because $Y$ is a $\mathbb
Q$-cohomology manifold.

\bigskip
(ii) $\implies$ (viii) Since  $Y$ is a $\mathbb
Q$-intersection cohomology manifold, combining (\ref{poincare}) with  Theorem
\ref{main1} we get:
$$
R\pi_*\mathbb Q_X\cong \mathbb Q_Y\oplus \mathcal H^{\bullet}\cong
\mathbb Q_Y\oplus \bigoplus_{k\geq 1} R^k\pi_*\mathbb Q_X[-k].
$$

\bigskip
(viii) $\implies$ (ii) See Remark \ref{unicity}, (i).

\bigskip
(ii) $\iff$ (ix) By \cite[Theorem 1.1]{Massey} we deduce that
$Y$ is a $\mathbb Q$-intersection cohomology manifold  if and only
if for any $y\in {\rm{Sing}}(Y)$ the link $\partial B_y$ has the
same $\mathbb Q$-homology type as a sphere $S^{2n-1}$. On the
other hand, via deformation to the normal cone, we may identify
$\partial B_y$ with the link of the vertex of the projective cone
over  $G_y \subseteq \mathbb P^{N-1}$. Restricting the Hopf bundle
$S^{2N-1}\to \mathbb P^{N-1}$ to $G_y$, we obtain an $S^1$-bundle
$\partial B_y\to G_y$ inducing the Thom-Gysin sequence \cite[p.
260]{Spanier}
$$
\dots\to H^k(G_y)\to H^k(\partial B_y)\to H^{k-1}(G_y)\to
H^{k+1}(G_y)\to H^{k+1}(\partial B_y)\to\dots
$$
And this sequence implies that $\partial B_y$ has the same
$\mathbb Q$-homology type as a sphere $S^{2n-1}$ if and only if
${H^{\bullet}(G_y)}\cong {H^{\bullet}(\mathbb P^{n-1})}$.

\begin{remark}\label{fibre'} By (\ref{fibre}) it follows that $h_2(G)\leq
h_{2n-2}(G)$. Therefore { if $Y$ is a $\mathbb Q$-cohomology
manifold then  $\dim G=0$ or $\dim G=n-1$}.
\end{remark}

\end{document}